\documentclass{amsart}
\pdfoutput=1

\usepackage{color}           
\usepackage[ascii]{inputenc} 
\usepackage{tikz}            
\usepackage{aliascnt}        
\usepackage{microtype}				 
\usepackage{graphics}        
\usepackage[bookmarksnumbered]{hyperref}
\usepackage[ngerman,USenglish]{babel} 
\usepackage[all]{hypcap}
\usepackage{soul} 

\renewcommand{\mathbb}{\mathbf}

\newcommand{\newjointcountertheorem}[3]{
	\newaliascnt{#1}{#2}
	\newtheorem{#1}[#1]{#3}
	\aliascntresetthe{#1}	
}

\newtheorem*{thm-cov}{Theorem \ref{covariance of the weight distribution}}
\newtheorem*{thm-wlln}{Theorem \ref{wlln}}
\newtheorem*{cnj-wlln}{Conjecture \ref{cnj-wlln}}

\newtheorem{thm}{Theorem}[section]
\newjointcountertheorem{satz}{thm}{Satz}
\newjointcountertheorem{lem}{thm}{Lemma}
\newjointcountertheorem{cor}{thm}{Corollary}
\newjointcountertheorem{prp}{thm}{Proposition}
\newjointcountertheorem{cnj}{thm}{Conjecture}
\newjointcountertheorem{que}{thm}{Question}
\newjointcountertheorem{fct}{thm}{Fact}
\theoremstyle{definition}
\newjointcountertheorem{dfn}{thm}{Definition}
\newjointcountertheorem{ntn}{thm}{Notation}
\newjointcountertheorem{rem}{thm}{Remark}
\newjointcountertheorem{nte}{thm}{Note}
\newjointcountertheorem{exl}{thm}{Example}

\DeclareMathOperator{\Cov}{Cov}

\DeclareMathOperator{\Meas}{Meas}
\DeclareMathOperator{\supp}{supp}
\newcommand{\Measc}{\Meas_\mathrm{c}} 
\DeclareMathOperator{\Var}{Var}

\newcommand{\qbinom}[2]{\genfrac{[}{]}{0pt}{}{#1}{#2}}

\def\Snospace~{\S{}}

\newcommand{\weakto}{\stackrel{\mathrm{w}}{\longrightarrow}}

\newcommand{\E}[2]{\operatorname{E}_{#1}[#2]}
\newcommand{\nicedot}{\raisebox{0.35ex}{\tikz \fill (0,0) circle (1pt);}} 
\newcommand{\annrel}[2]{\stackrel{\makebox[0pt]{\scriptsize #1}}{#2}} 
\renewcommand{\phi}{\varphi}
\renewcommand{\epsilon}{\varepsilon}

\graphicspath{{images/}}
\numberwithin{equation}{section}
\allowdisplaybreaks[4] 

\hypersetup{pdfauthor={Thomas Bliem and Stavros Kousidis},pdftitle={On the Law of Large Numbers for Demazure Modules of sl2hat}}

\begin{document}

	\title[Law of large numbers for Demazure modules]{On the Law of Large Numbers for \\Demazure Modules of $\widehat{\mathfrak{sl}}_2$}

	\author[Thomas Bliem]{Thomas Bliem}
  	\address{Thomas Bliem, Department of Mathematics\\San Francisco State University\\1600 Holloway Ave\\San Francisco CA 94109\\United States}
  	\email{\href{mailto:bliem@math.sfsu.edu}{bliem@math.sfsu.edu}}

  	\author[Stavros Kousidis]{Stavros Kousidis}
  	\address{\selectlanguage{ngerman}Stavros Kousidis, Mathematisches Institut, Universit\"at zu K\"oln, Weyertal 86--90, 50931 K\"oln, Germany}
	\email{\href{mailto:st.kousidis@googlemail.com}{st.kousidis@googlemail.com}}

	\subjclass[2010]{Primary 06B15 Representation theory, Secondary 60B99 Probability theory on algebraic and topological structures.}
	\keywords{affine Kac-Moody algebra, Demazure module, variance, covariance, law of large numbers}
	
	\begin{abstract}
		We determine the covariance of the weight distribution in level $1$ Demazure modules of $\widehat{\mathfrak{sl}}_2$. This allows us to prove a weak law of large numbers for these weight distributions, and leads to a conjecture about the asymptotic concentration of weights for arbitrary Demazure modules.
	\end{abstract}

	\maketitle


\section{Introduction}
	
	The motivation for this work is to understand the qualitative features of dimensions of weight spaces in Demazure modules $V_w(\Lambda)$ as the highest weight $\Lambda$ is fixed and the length of the Weyl group element $w$ becomes large.
	We want to be able to answer questions like:  Where is the bulk of weights, counted with their individual multiplicities, concentrated?  Do they tend to spread out or do they concentrate in a certain part of the Weyl polytope? Etc.

	Even though there is a large body of work on Demazure modules, these kind of questions are still difficult to answer.
	We believe that in order to answer them, one needs to look at weight distributions with a probabilistic eye, starting with the computation of the expected value and the covariance, possibly followed by skewness and kurtosis later.
	We determined the expected value of the weight distribution of any Demazure module of $\widehat{\mathfrak{sl}}_2$ in \cite{bk10}.
	The next step is to determine the covariance, which we complete in the present article for the modules in level $1$:
	
	\begin{thm-cov}
		Let $j \in \{0, 1\}$ and $w \in W^\mathrm{aff}$ such that the length $l(ws_j) < l(w) = N$.	
	Then the covariance matrix $\Sigma$ of the degree $\langle -d, \nicedot \rangle$ and the finite weight $\langle \alpha_1^\vee, \nicedot \rangle$ in $V_{w}(\Lambda_j)$ is given by
	\[
		\Sigma = \begin{cases}
			\begin{pmatrix}
				\frac{N(N-1)(2N+5)}{96} & 0 \\
			  0                       & N
			\end{pmatrix}
			& \text{if $N \equiv j \mod (2)$,} \\
			\begin{pmatrix}
				\frac{N(N-1)(2N+5)}{96} + \frac N4 &	 \frac N2\\
				\frac N2                           & N
			\end{pmatrix}
			& \text{if $N \not\equiv j \mod (2)$.}
		\end{cases}			
	\]
	\end{thm-cov}

	See \autoref{example figures} for some examples.
	\begin{figure}
		\begin{tabular}{cccc}
			\vtop{\null{\hbox{
				\begin{tikzpicture}[x=1.5em,y=-1.5em]
					\draw[thick] (0,13.75) ellipse (1.5811 and 4.8412);
					\draw (0,-2) node{finite weight};
					\draw (-7,12.5) node[rotate=90]{degree};
					\foreach \x in {-10,-8,...,10} \draw (\x/2, -1) node{\x}; 
					\foreach \x in {0,...,25} \draw (-6,\x) node{$\x$}; 
					\draw (0,0) node[fill=white,inner sep=0.2ex]{$1$};
					\draw (-1,1) node[fill=white,inner sep=0.2ex]{$1$};
					\draw (0,1) node[fill=white,inner sep=0.2ex]{$1$};
					\draw (1,1) node[fill=white,inner sep=0.2ex]{$1$};
					\draw (-1,2) node[fill=white,inner sep=0.2ex]{$1$};
					\draw (0,2) node[fill=white,inner sep=0.2ex]{$2$};
					\draw (1,2) node[fill=white,inner sep=0.2ex]{$1$};
					\draw (-1,3) node[fill=white,inner sep=0.2ex]{$2$};
					\draw (0,3) node[fill=white,inner sep=0.2ex]{$3$};
					\draw (1,3) node[fill=white,inner sep=0.2ex]{$2$};
					\draw (-2,4) node[fill=white,inner sep=0.2ex]{$1$};
					\draw (-1,4) node[fill=white,inner sep=0.2ex]{$3$};
					\draw (0,4) node[fill=white,inner sep=0.2ex]{$5$};
					\draw (1,4) node[fill=white,inner sep=0.2ex]{$3$};
					\draw (2,4) node[fill=white,inner sep=0.2ex]{$1$};
					\draw (-2,5) node[fill=white,inner sep=0.2ex]{$1$};
					\draw (-1,5) node[fill=white,inner sep=0.2ex]{$5$};
					\draw (0,5) node[fill=white,inner sep=0.2ex]{$7$};
					\draw (1,5) node[fill=white,inner sep=0.2ex]{$5$};
					\draw (2,5) node[fill=white,inner sep=0.2ex]{$1$};
					\draw (-2,6) node[fill=white,inner sep=0.2ex]{$2$};
					\draw (-1,6) node[fill=white,inner sep=0.2ex]{$6$};
					\draw (0,6) node[fill=white,inner sep=0.2ex]{$9$};
					\draw (1,6) node[fill=white,inner sep=0.2ex]{$6$};
					\draw (2,6) node[fill=white,inner sep=0.2ex]{$2$};
					\draw (-2,7) node[fill=white,inner sep=0.2ex]{$3$};
					\draw (-1,7) node[fill=white,inner sep=0.2ex]{$9$};
					\draw (0,7) node[fill=white,inner sep=0.2ex]{$11$};
					\draw (1,7) node[fill=white,inner sep=0.2ex]{$9$};
					\draw (2,7) node[fill=white,inner sep=0.2ex]{$3$};
					\draw (-2,8) node[fill=white,inner sep=0.2ex]{$4$};
					\draw (-1,8) node[fill=white,inner sep=0.2ex]{$10$};
					\draw (0,8) node[fill=white,inner sep=0.2ex]{$14$};
					\draw (1,8) node[fill=white,inner sep=0.2ex]{$10$};
					\draw (2,8) node[fill=white,inner sep=0.2ex]{$4$};
					\draw (-3,9) node[fill=white,inner sep=0.2ex]{$1$};
					\draw (-2,9) node[fill=white,inner sep=0.2ex]{$5$};
					\draw (-1,9) node[fill=white,inner sep=0.2ex]{$13$};
					\draw (0,9) node[fill=white,inner sep=0.2ex]{$16$};
					\draw (1,9) node[fill=white,inner sep=0.2ex]{$13$};
					\draw (2,9) node[fill=white,inner sep=0.2ex]{$5$};
					\draw (3,9) node[fill=white,inner sep=0.2ex]{$1$};
					\draw (-3,10) node[fill=white,inner sep=0.2ex]{$1$};
					\draw (-2,10) node[fill=white,inner sep=0.2ex]{$7$};
					\draw (-1,10) node[fill=white,inner sep=0.2ex]{$14$};
					\draw (0,10) node[fill=white,inner sep=0.2ex]{$18$};
					\draw (1,10) node[fill=white,inner sep=0.2ex]{$14$};
					\draw (2,10) node[fill=white,inner sep=0.2ex]{$7$};
					\draw (3,10) node[fill=white,inner sep=0.2ex]{$1$};
					\draw (-3,11) node[fill=white,inner sep=0.2ex]{$2$};
					\draw (-2,11) node[fill=white,inner sep=0.2ex]{$8$};
					\draw (-1,11) node[fill=white,inner sep=0.2ex]{$16$};
					\draw (0,11) node[fill=white,inner sep=0.2ex]{$19$};
					\draw (1,11) node[fill=white,inner sep=0.2ex]{$16$};
					\draw (2,11) node[fill=white,inner sep=0.2ex]{$8$};
					\draw (3,11) node[fill=white,inner sep=0.2ex]{$2$};
					\draw (-3,12) node[fill=white,inner sep=0.2ex]{$2$};
					\draw (-2,12) node[fill=white,inner sep=0.2ex]{$9$};
					\draw (-1,12) node[fill=white,inner sep=0.2ex]{$16$};
					\draw (0,12) node[fill=white,inner sep=0.2ex]{$20$};
					\draw (1,12) node[fill=white,inner sep=0.2ex]{$16$};
					\draw (2,12) node[fill=white,inner sep=0.2ex]{$9$};
					\draw (3,12) node[fill=white,inner sep=0.2ex]{$2$};
					\draw (-3,13) node[fill=white,inner sep=0.2ex]{$3$};
					\draw (-2,13) node[fill=white,inner sep=0.2ex]{$10$};
					\draw (-1,13) node[fill=white,inner sep=0.2ex]{$18$};
					\draw (0,13) node[fill=white,inner sep=0.2ex]{$20$};
					\draw (1,13) node[fill=white,inner sep=0.2ex]{$18$};
					\draw (2,13) node[fill=white,inner sep=0.2ex]{$10$};
					\draw (3,13) node[fill=white,inner sep=0.2ex]{$3$};
					\draw (-3,14) node[fill=white,inner sep=0.2ex]{$3$};
					\draw (-2,14) node[fill=white,inner sep=0.2ex]{$10$};
					\draw (-1,14) node[fill=white,inner sep=0.2ex]{$16$};
					\draw (0,14) node[fill=white,inner sep=0.2ex]{$19$};
					\draw (1,14) node[fill=white,inner sep=0.2ex]{$16$};
					\draw (2,14) node[fill=white,inner sep=0.2ex]{$10$};
					\draw (3,14) node[fill=white,inner sep=0.2ex]{$3$};
					\draw (-3,15) node[fill=white,inner sep=0.2ex]{$4$};
					\draw (-2,15) node[fill=white,inner sep=0.2ex]{$10$};
					\draw (-1,15) node[fill=white,inner sep=0.2ex]{$16$};
					\draw (0,15) node[fill=white,inner sep=0.2ex]{$18$};
					\draw (1,15) node[fill=white,inner sep=0.2ex]{$16$};
					\draw (2,15) node[fill=white,inner sep=0.2ex]{$10$};
					\draw (3,15) node[fill=white,inner sep=0.2ex]{$4$};
					\draw (-4,16) node[fill=white,inner sep=0.2ex]{$1$};
					\draw (-3,16) node[fill=white,inner sep=0.2ex]{$4$};
					\draw (-2,16) node[fill=white,inner sep=0.2ex]{$10$};
					\draw (-1,16) node[fill=white,inner sep=0.2ex]{$14$};
					\draw (0,16) node[fill=white,inner sep=0.2ex]{$16$};
					\draw (1,16) node[fill=white,inner sep=0.2ex]{$14$};
					\draw (2,16) node[fill=white,inner sep=0.2ex]{$10$};
					\draw (3,16) node[fill=white,inner sep=0.2ex]{$4$};
					\draw (4,16) node[fill=white,inner sep=0.2ex]{$1$};
					\draw (-4,17) node[fill=white,inner sep=0.2ex]{$1$};
					\draw (-3,17) node[fill=white,inner sep=0.2ex]{$5$};
					\draw (-2,17) node[fill=white,inner sep=0.2ex]{$9$};
					\draw (-1,17) node[fill=white,inner sep=0.2ex]{$13$};
					\draw (0,17) node[fill=white,inner sep=0.2ex]{$14$};
					\draw (1,17) node[fill=white,inner sep=0.2ex]{$13$};
					\draw (2,17) node[fill=white,inner sep=0.2ex]{$9$};
					\draw (3,17) node[fill=white,inner sep=0.2ex]{$5$};
					\draw (4,17) node[fill=white,inner sep=0.2ex]{$1$};
					\draw (-4,18) node[fill=white,inner sep=0.2ex]{$1$};
					\draw (-3,18) node[fill=white,inner sep=0.2ex]{$4$};
					\draw (-2,18) node[fill=white,inner sep=0.2ex]{$8$};
					\draw (-1,18) node[fill=white,inner sep=0.2ex]{$10$};
					\draw (0,18) node[fill=white,inner sep=0.2ex]{$11$};
					\draw (1,18) node[fill=white,inner sep=0.2ex]{$10$};
					\draw (2,18) node[fill=white,inner sep=0.2ex]{$8$};
					\draw (3,18) node[fill=white,inner sep=0.2ex]{$4$};
					\draw (4,18) node[fill=white,inner sep=0.2ex]{$1$};
					\draw (-4,19) node[fill=white,inner sep=0.2ex]{$1$};
					\draw (-3,19) node[fill=white,inner sep=0.2ex]{$4$};
					\draw (-2,19) node[fill=white,inner sep=0.2ex]{$7$};
					\draw (-1,19) node[fill=white,inner sep=0.2ex]{$9$};
					\draw (0,19) node[fill=white,inner sep=0.2ex]{$9$};
					\draw (1,19) node[fill=white,inner sep=0.2ex]{$9$};
					\draw (2,19) node[fill=white,inner sep=0.2ex]{$7$};
					\draw (3,19) node[fill=white,inner sep=0.2ex]{$4$};
					\draw (4,19) node[fill=white,inner sep=0.2ex]{$1$};
					\draw (-4,20) node[fill=white,inner sep=0.2ex]{$1$};
					\draw (-3,20) node[fill=white,inner sep=0.2ex]{$3$};
					\draw (-2,20) node[fill=white,inner sep=0.2ex]{$5$};
					\draw (-1,20) node[fill=white,inner sep=0.2ex]{$6$};
					\draw (0,20) node[fill=white,inner sep=0.2ex]{$7$};
					\draw (1,20) node[fill=white,inner sep=0.2ex]{$6$};
					\draw (2,20) node[fill=white,inner sep=0.2ex]{$5$};
					\draw (3,20) node[fill=white,inner sep=0.2ex]{$3$};
					\draw (4,20) node[fill=white,inner sep=0.2ex]{$1$};
					\draw (-4,21) node[fill=white,inner sep=0.2ex]{$1$};
					\draw (-3,21) node[fill=white,inner sep=0.2ex]{$3$};
					\draw (-2,21) node[fill=white,inner sep=0.2ex]{$4$};
					\draw (-1,21) node[fill=white,inner sep=0.2ex]{$5$};
					\draw (0,21) node[fill=white,inner sep=0.2ex]{$5$};
					\draw (1,21) node[fill=white,inner sep=0.2ex]{$5$};
					\draw (2,21) node[fill=white,inner sep=0.2ex]{$4$};
					\draw (3,21) node[fill=white,inner sep=0.2ex]{$3$};
					\draw (4,21) node[fill=white,inner sep=0.2ex]{$1$};
					\draw (-4,22) node[fill=white,inner sep=0.2ex]{$1$};
					\draw (-3,22) node[fill=white,inner sep=0.2ex]{$2$};
					\draw (-2,22) node[fill=white,inner sep=0.2ex]{$3$};
					\draw (-1,22) node[fill=white,inner sep=0.2ex]{$3$};
					\draw (0,22) node[fill=white,inner sep=0.2ex]{$3$};
					\draw (1,22) node[fill=white,inner sep=0.2ex]{$3$};
					\draw (2,22) node[fill=white,inner sep=0.2ex]{$3$};
					\draw (3,22) node[fill=white,inner sep=0.2ex]{$2$};
					\draw (4,22) node[fill=white,inner sep=0.2ex]{$1$};
					\draw (-4,23) node[fill=white,inner sep=0.2ex]{$1$};
					\draw (-3,23) node[fill=white,inner sep=0.2ex]{$2$};
					\draw (-2,23) node[fill=white,inner sep=0.2ex]{$2$};
					\draw (-1,23) node[fill=white,inner sep=0.2ex]{$2$};
					\draw (0,23) node[fill=white,inner sep=0.2ex]{$2$};
					\draw (1,23) node[fill=white,inner sep=0.2ex]{$2$};
					\draw (2,23) node[fill=white,inner sep=0.2ex]{$2$};
					\draw (3,23) node[fill=white,inner sep=0.2ex]{$2$};
					\draw (4,23) node[fill=white,inner sep=0.2ex]{$1$};
					\draw (-4,24) node[fill=white,inner sep=0.2ex]{$1$};
					\draw (-3,24) node[fill=white,inner sep=0.2ex]{$1$};
					\draw (-2,24) node[fill=white,inner sep=0.2ex]{$1$};
					\draw (-1,24) node[fill=white,inner sep=0.2ex]{$1$};
					\draw (0,24) node[fill=white,inner sep=0.2ex]{$1$};
					\draw (1,24) node[fill=white,inner sep=0.2ex]{$1$};
					\draw (2,24) node[fill=white,inner sep=0.2ex]{$1$};
					\draw (3,24) node[fill=white,inner sep=0.2ex]{$1$};
					\draw (4,24) node[fill=white,inner sep=0.2ex]{$1$};
					\draw (-5,25) node[fill=white,inner sep=0.2ex]{$1$};
					\draw (-4,25) node[fill=white,inner sep=0.2ex]{$1$};
					\draw (-3,25) node[fill=white,inner sep=0.2ex]{$1$};
					\draw (-2,25) node[fill=white,inner sep=0.2ex]{$1$};
					\draw (-1,25) node[fill=white,inner sep=0.2ex]{$1$};
					\draw (0,25) node[fill=white,inner sep=0.2ex]{$1$};
					\draw (1,25) node[fill=white,inner sep=0.2ex]{$1$};
					\draw (2,25) node[fill=white,inner sep=0.2ex]{$1$};
					\draw (3,25) node[fill=white,inner sep=0.2ex]{$1$};
					\draw (4,25) node[fill=white,inner sep=0.2ex]{$1$};
					\draw (5,25) node[fill=white,inner sep=0.2ex]{$1$};
				\end{tikzpicture}
			}}} & \vtop{\null{\hbox{
				\begin{tikzpicture}[x=.96bp,y=-.96bp]
					\pgftext[top]{\includegraphics{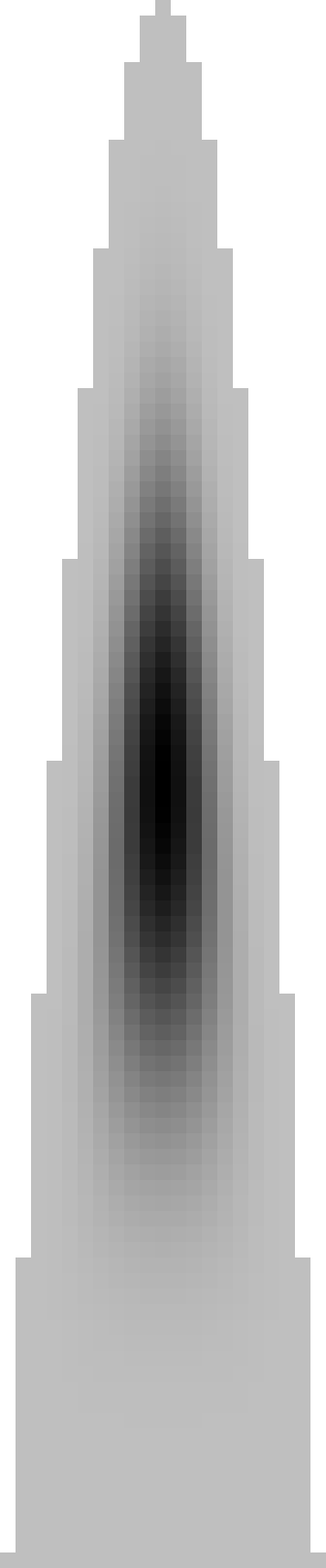}};
					\draw[white] (0,53) ellipse (2.2361 and 13.3463);
				\end{tikzpicture}
			}}} & \vtop{\null{\hbox{
				\begin{tikzpicture}[x=.96bp,y=-.96bp]
					\pgftext[top]{\includegraphics{30}};
					\draw[white] (0,116.75) ellipse (2.7386 and 24.2706);
				\end{tikzpicture}
			}}} & \vtop{\null{\hbox{
				\begin{tikzpicture}[x=.96bp,y=-.96bp]
					\pgftext[top]{\includegraphics{40}};
					\draw[white] (0,205.5) ellipse (3.1623 and 37.1652);
				\end{tikzpicture}
			}}}
		\end{tabular}
		\caption{Weight distribution of $V_{(s_1 s_0)^k}(\Lambda_0)$ for $k = 5$ (indicated by numbers) and $k = 10, 15, 20$ (indicated by shades of gray).
		The covariances are represented by covariance ellipses.}
		\label{example figures}
	\end{figure}
	Our \autoref{covariance of the weight distribution} is enough to obtain the weak law of large numbers as a corollary, to our knowledge the first result to give an idea about the overall weight distribution in Demazure modules as the length of the Weyl group element becomes large:

	\begin{thm-wlln}
		Let $(\Lambda^{(k)})$ be a sequence in $\{\Lambda_0, \Lambda_1\}$, and $(w^{(k)})$ a sequence in $W^\mathrm{aff}$ such that $l(w^{(k)}) \to \infty$.
		Let $\tilde\mu^{(k)} \in \Meas(\mathbf{R}^2)$ be the joint distribution of the degree and the finite weight in $V_{w^{(k)}}(\Lambda^{(k)})$, normalized to a probability distribution and rescaled individually in the two coordinates such that $\supp(\tilde\mu^{(k)})$ just fits into the rectangle $[0,1] \times [-1,1]$.
		Then	, as $k \to \infty$,
		\[
			\tilde \mu^{(k)} \weakto \delta_{(\frac 12, 0)} .
		\]
	\end{thm-wlln}
	
	Some examples of $\tilde\mu^{(k)}$ are shown in \autoref{fig wlln}. Furthermore, based on our present result and \cite[Corollary 4.3]{bk10} we propose the following conjecture:
	
	\begin{cnj-wlln}
	Fix a dominant integral weight $\Lambda$ and a sequence $(w^{(k)})$ in $W^{\mathrm{aff}}$ such that $l(w^{(k)}) \to \infty$.
	Let $\mu^{(k)} \in \Meas(\mathbf{N} \times \mathbf{Z})$ be the joint distribution of the degree and the finite weight in $V_{w^{(k)}}(\Lambda)$.
		Let $\tilde\mu^{(k)} \in \Meas(\mathbf{R}^2)$ be the distribution obtained from $\mu^{(k)}$ by normalizing to a probability distribution and rescaling the two coordinates individually so that $\supp(\tilde\mu^{(k)})$ just fits into the rectangle $[0,1] \times [-1,1]$.
		Then, as $k \to \infty$,
		\[
			\tilde\mu^{(k)} \weakto \delta_{\left( \frac{\langle c, \Lambda \rangle +2}{3(\langle c, \Lambda \rangle+1)}, 0 \right)} ,
		\]
		where $c = \alpha_0^\vee + \alpha_1^\vee$ denotes the canonical central element.\end{cnj-wlln}

	\begin{figure}
		\begin{tikzpicture}[x=4em,y=-4em]
			\pgftext[top]{\includegraphics[width=8em,height=4em]{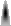}};
			\draw (-1.1,0) -- (-1.3,0) node[left]{$0$};
			\draw (-1.1,1) -- (-1.3,1) node[left]{$1$};
			\draw (-1,-0.1) -- (-1,-0.3) node[above]{$-1$};
			\draw (1,-0.1) -- (1,-0.3) node[above]{$1$};
			\draw (-1,0) -- (1,0) -- (1,1) -- (-1,1) -- cycle;
		\end{tikzpicture}
		\quad
		\begin{tikzpicture}[x=4em,y=-4em]
			\pgftext[top]{\includegraphics[width=8em,height=4em]{20.png}};
			\draw (-1.1,0) -- (-1.3,0) node[left]{$0$};
			\draw (-1.1,1) -- (-1.3,1) node[left]{$1$};
			\draw (-1,-0.1) -- (-1,-0.3) node[above]{$-1$};
			\draw (1,-0.1) -- (1,-0.3) node[above]{$1$};
			\draw (-1,0) -- (1,0) -- (1,1) -- (-1,1) -- cycle;
		\end{tikzpicture}
		\quad
		\begin{tikzpicture}[x=4em,y=-4em]
			\pgftext[top]{\includegraphics[width=8em,height=4em]{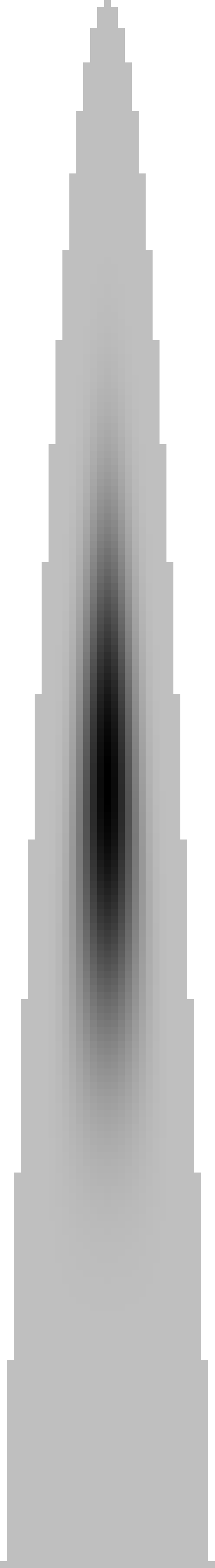}};
			\draw (-1.1,0) -- (-1.3,0) node[left]{$0$};
			\draw (-1.1,1) -- (-1.3,1) node[left]{$1$};
			\draw (-1,-0.1) -- (-1,-0.3) node[above]{$-1$};
			\draw (1,-0.1) -- (1,-0.3) node[above]{$1$};
			\draw (-1,0) -- (1,0) -- (1,1) -- (-1,1) -- cycle;
		\end{tikzpicture}
		\quad
		\begin{tikzpicture}[x=4em,y=-4em]
			\pgftext[top]{\includegraphics[width=8em,height=4em]{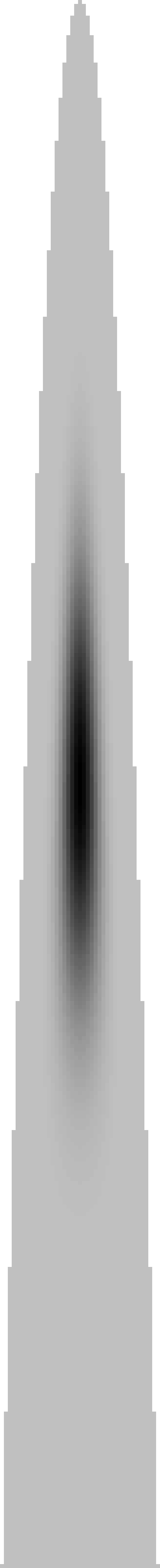}};
			\draw (-1.1,0) -- (-1.3,0) node[left]{$0$};
			\draw (-1.1,1) -- (-1.3,1) node[left]{$1$};
			\draw (-1,-0.1) -- (-1,-0.3) node[above]{$-1$};
			\draw (1,-0.1) -- (1,-0.3) node[above]{$1$};
			\draw (-1,0) -- (1,0) -- (1,1) -- (-1,1) -- cycle;
		\end{tikzpicture}
		\caption{Weight distributions of $V_{(s_1 s_0)^k}(\Lambda_0)$ for $k = 5, 10, 15, 20$, rescaled into the rectangle $[0,1] \times [-1,1]$.}
		\label{fig wlln}
	\end{figure}
	
	The outline of the proof of \autoref{covariance of the weight distribution} and the organization of the paper is as follows:
	It is easy to show that the degree and the finite weight are uncorrelated.
	The covariance of the finite weight can be directly obtained from a known result by Sanderson about the real characters of Demazure modules \cite{MR1407880}.
	Hence the main part is about the variance of the degree distribution.
	The strategy is to proceed by induction along Demazure's character formula.
	In \autoref{sec:technical results}, we show that a single recursion step expresses the second moment of the degree distribution of a given Demazure module in terms of the \emph{third} moments of the weight distribution of a smaller Demazure module.
	We try to express these third moments in terms of the (known) third moments of the distribution of the finite weight, but succeed not quite (\autoref{specialized fake recursion}).
	In \autoref{sec:palindromicity}, we show that the weight multiplicities are symmetric in each string of weights differing only by a multiple of the null root $\delta$ (\autoref{strings are symmetric}).
	We use this in \autoref{sec:stretching} to show that the covariance between two specific quadratic functions vanishes (\autoref{intermediate covariance}).
	This allows us to explicitly compute the previously problematic third moments (\autoref{determine linear form}).
	This yields an explicit recurrence relation of a purely additive nature (\autoref{modified recurrence relation}), which is easy to solve (\autoref{degree variance}).
	We complement our result in \autoref{covariance section} by determining the full covariance matrix (\autoref{covariance of the weight distribution}).
	In \autoref{sec:wlln} we deduce a weak law of large numbers for our Demazure modules (\autoref{wlln}) and formulate a conjecture for the general case (\autoref{cnj-wlln}).

\section{Notation}
\label{sec:notation}
	
	For general notation about Kac--Moody algebras we mostly follow Kac \cite{MR1104219}.
	Let $\mathfrak{h}$ be a $3$-dimensional complex vector space and  $\alpha_0^\vee, \alpha_1^\vee \in \mathfrak{h}$, $\alpha_0, \alpha_1 \in \mathfrak{h}^*$ a realization of the generalized Cartan matrix $A = \left( \begin{smallmatrix} 2 & -2 \\ -2 & 2 \end{smallmatrix} \right)$.
	Let $\widehat{\mathfrak{sl}}_2 = \mathfrak{g}(A)$ be the associated affine Lie algebra.
	We denote the canonical pairing between $\mathfrak{h}$ and $\mathfrak{h}^*$ by $\langle h, \alpha \rangle = \alpha(h)$.
	Choose $d \in \mathfrak{h}$ such that $\langle d, \alpha_0 \rangle = 1$ and $\langle d, \alpha_1 \rangle = 0$.
	Such an element $d$ is called a \emph{scaling element}.
	The \emph{degree} of $\lambda \in \mathfrak{h}^*$ is defined as $\langle -d, \lambda \rangle$.
	The set $\{ \alpha_0^\vee, \alpha_1^\vee, d \}$ is a basis of $\mathfrak{h}$.
	Let $\{ \Lambda_0, \Lambda_1, \delta \}$ be the corresponding dual basis of $\mathfrak{h}^*$.
	Then $\delta = \alpha_0 + \alpha_1$, and $\Lambda_0$, $\Lambda_1$ are called \emph{fundamental weights}.
	
	Integrable highest weight modules of $\widehat{\mathfrak{sl}}_2$ are parametrized up to isomorphism by dominant integral weights
	\[
		\Lambda = m \Lambda_0 + n \Lambda_1 + z \delta
	\]
	for $m, n \in \mathbf{N}$ and $z \in \mathbf{C}$.
	We denote the integrable highest weight module corresponding to $\Lambda$ by $V(\Lambda)$.
	As a change in $z$ simply corresponds to the choice of a different scaling element $d \in \mathfrak{h}$, it is customary to suppose $z = 0$ and only consider dominant integral weights of the form $\Lambda = m \Lambda_0 + n \Lambda_1$, which we do from now on.
	
	Let $c = \alpha_0^\vee + \alpha_1^\vee$ be the canonical central element.
	The \emph{level} of a weight $\lambda \in \mathfrak{h}^*$ is defined as $\langle c, \lambda \rangle$.
	If $\lambda = m\Lambda_0 + n\Lambda_1$, then its level is $\langle c, \lambda \rangle = m + n$.
	Hence the dominant integral weights of level $1$ are exactly the fundamental weights $\Lambda_0$, $\Lambda_1$.
	
	A weight $\lambda \in \mathfrak{h}^*$ is said to \emph{occur} in a given integrable highest weight module $V(\Lambda)$ if the weight space $V(\Lambda)_\lambda$ is nontrivial.
	The set of weights occuring in $V(\Lambda)$ is contained in the (affine) lattice
	\[
		\Gamma = \Lambda + \mathbf{Z}\alpha_0 + \mathbf{Z} \alpha_1 .
	\]
	We define coordinates $a, b$ on $\Gamma$ by
	\[
		\lambda = \Lambda - a(\lambda)\alpha_0 - b(\lambda)\alpha_1
	\]
	for all $\lambda \in \Gamma$.
	Note that $a, b$ depend on $\Lambda$.
	In \autoref{example figures}\footnote{The software used to produce most of the figures in this article is available at \url{http://sourceforge.net/projects/demazure}.}, each matrix component resp.\ pixel represents a point in $\Gamma$.
	We write $\Gamma_j = \Lambda_j + \mathbf{Z}\alpha_0 + \mathbf{Z} \alpha_1$ for $j\in \{ 0, 1\}$ to refer to the two lattices corresponding to the fundamental weights. 
		
	For $j \in \{0, 1\}$ define linear maps $s_j : \mathfrak{h}^* \to \mathfrak{h}^*$ by
	\[
		s_j(\lambda) = \lambda - \langle \alpha_j^\vee, \lambda \rangle \alpha_j .
	\]
	The Weyl group $W^{\mathrm{aff}}$ of $\widehat{\mathfrak{sl}}_2$ is by definition the subgroup of $\mathrm{GL}(\mathfrak{h}^*)$ generated by $s_0$ and $s_1$.
	All elements of $W^{\mathrm{aff}}$ have the form
	\[
		w_{N,0} = \underbrace{\cdots s_0 s_1 s_0}_{\text{$N$ factors}}
		\quad \text{or} \quad
		w_{N,1} = \underbrace{\cdots s_1 s_0 s_1}_{\text{$N$ factors}}
	\]
	for $N \geq 0$.
	We abbreviate $w_N = w_{N,0}$, as these are the elements we will mostly discuss.
	Write $l(w)$ for the length of a reduced decomposition of $w \in W^{\mathrm{aff}}$, so $l(w_{N,j}) = N$.
	
	Let $\mathfrak{n}_+ \subset \widehat{\mathfrak{sl}}_2$ be the sum of the positive root spaces.
	For $w \in W^\mathrm{aff}$ and $\Lambda$ a dominant integral weight, define the \emph{Demazure module} $V_w(\Lambda)$ \cite{MR862200,MR894387,MR980506} to be the $(\mathfrak{h} \oplus \mathfrak{n}_+)$-module generated by $V(\Lambda)_{w\Lambda}$.
	As $V_w(\Lambda)$ is in particular an $\mathfrak{h}$-module, it has a weight space decomposition
	\[
		V_w(\Lambda) = \bigoplus_{\lambda \in \mathfrak{h}^*} V_w(\Lambda)_\lambda .
	\]
	Let $\Measc(\Gamma)$ denote the set of measures on $\Gamma$ with compact (hence finite) support. We define the \emph{weight distribution} of $V_{w_{N,j}}(\Lambda)$ to be
	\[
		\mu_{N,j}
		= \sum_{\lambda \in \mathfrak{h}^*} \dim(V_{w_{N,j}}(\Lambda)_\lambda) \cdot \delta_\lambda
		\in \Measc(\Gamma)
	\]
	and again abbreviate $\mu_N = \mu_{N,0}$ for the case mostly considered.
	Note that the dependence on $\Lambda$ is important, but only implicit in the notation.
	
	Given $\Lambda$, define operators $D_0, D_1$ on the space $\Measc^\pm(\Gamma)$ of signed measures on $\Gamma$ with compact support by 
	\[
		D_j \delta_\lambda 
		= \sum_{i=0}^{\langle \alpha_j^\vee, \lambda \rangle} \delta_{\lambda - i \alpha_j}
 	\]
	for $j \in \{0, 1\}$ and $\lambda \in \Gamma$.
	Here we use the conventions that
		$\sum_{i=0}^{-1} a_i = 0$
	and
		$\sum_{i=0}^k a_i = - \sum_{i=k+1}^{-1} a_i$
	for $k < -1$, hence the necessity to consider signed measures.
	Demazure's character formula for Kac--Moody algebras \cite{MR862200,MR894387,MR980506} states that
	\[
		\mu_{N,0} = \underbrace{\cdots D_0 D_1 D_0}_{\text{$N$ factors}} \delta_\Lambda
		\quad \text{and} \quad
		\mu_{N,1} = \underbrace{\cdots D_1 D_0 D_1}_{\text{$N$ factors}} \delta_\Lambda .
	\]
	Consider elements of $\mathfrak{h}$ as functions on $\mathfrak{h}^*$.
	We refer to the push-forward measure $(-d)_* \mu_{N,j}$ as the \emph{degree distribution}, and to $(\alpha_1^\vee)_* \mu_{N,j}$ as the \emph{distribution of the finite weight} of $V_{w_{N,j}}(\Lambda)$.
	Note that in terms of the coordinates $a$, $b$ we have that as functions on $\Gamma_i$
	\[
		-d = a
		\quad \text{and} \quad
		\alpha_1^\vee = \begin{cases}
			-2(a-b)     & \text{if $j = 0$,} \\
			-2(a-b) + 1 & \text{if $j = 1$.}
		\end{cases}
	\]
	Hence the degree distribution is $a_* \mu_{N,j}$, and the distribution of the finite weight is $(a-b)_* \mu_{N,j}$ up to translation and scaling.

	Recall that the expected value of a function $f : \Gamma \to \mathbf{R}$ with respect to a nonzero measure $\mu \in \Measc(\Gamma)$ is
	\[
		\E{\mu}{f} = \frac{1}{\mu(\Gamma)} \sum_{\lambda \in \Gamma} \mu(\{\lambda\}) f(\lambda).
	\]
	The covariance of two functions $f$ and $g$ is
	\[
		\Cov_\mu(f, g) = \E{\mu}{(f - \E{\mu}{f})(g - \E{\mu}{g})},
	\]
	and the variance of $f$ is $\Var_\mu(f) = \Cov_\mu(f,f)$.

\section{Almost a recursion formula}
\label{sec:technical results}

	The characters of Demazure modules for the highest weight $\Lambda_1$ are easily obtained from the ones for highest weight $\Lambda_0$ by symmetry of the Dynkin diagram, so we start by only discussing the latter.
	As $V_{s_1}(\Lambda_0) = \mathbf{C}$ we only need to consider Weyl group elements of the form $w_N = \cdots s_0 s_1 s_0$ ($N$ factors).
		
	\begin{prp}
	\label{specialized fake recursion}
		Let $\mu_N$ be the weight distribution of the Demazure module $V_{w_N}(\Lambda_0)$. Then,
		\begin{align*}
			\E{\mu_{N+1}}{a^2} & = \E{\mu_N}{a^2} + \frac{N(N^2 +N+2)}{16} + 2 \Cov_{\mu_N}(b,(a-b)^2)
			&& \text{for odd $N$},  \\
			\E{\mu_{N+1}}{b^2} & = \E{\mu_N}{b^2} + \frac{N^2(N+1)}{16} + 2 \Cov_{\mu_N}(a,(a-b)^2) 
			&& \text{for even $N$}.
		\end{align*}
	\end{prp}

	\begin{proof}
		We only consider the case of odd $N$, the other case being similar.
		Then $\mu_{N+1} = D_1 \mu_N = D_1D_0 \mu_{N-1}$.
		We will use Sanderson's formula for the real character of a Demazure module \cite{MR1407880}, which implies that the distribution of $a-b$ with respect to $\mu_N$ is given by
		\begin{equation} \label{sanderson}
			\left( a - b + \tfrac {N-1}2  \right)_* \mu_N
			= 2^N B(N,\tfrac 12) ,
		\end{equation}
		where $B(N,\tfrac 12)$ is the binomial distribution for $N$ trials with success probability $\frac 12$.

		By unwinding Demazure's character formula \cite[Lemma 3.6]{bk10}, we obtain
		\begin{align*}
			\E{\mu_{N+1}}{a^2}
			&= \frac{\mu_N (\Gamma)}{\mu_{N+1} (\Gamma)} \Bigl( 2\E{\mu_N}{a^2} + 2\Cov_{\mu_N}(a^2,a-b) \Bigr)
			.
		\end{align*}
		\noindent As $\tfrac{\mu_N (\Gamma)}{\mu_{N+1} (\Gamma)} = \tfrac 12$ by \eqref{sanderson},
		
		\begin{align*}
			\E{\mu_{N+1}}{a^2}
			&= \E{\mu_N}{a^2} + \Cov_{\mu_N}(a^2,a-b) \\
			&= \E{\mu_N}{a^2} + \Cov_{\mu_N} \Bigl( (a-b)^2 + 2ab - b^2, a-b \Bigr) \\
			&= \E{\mu_N}{a^2}
			   + \underbrace{\Cov_{\mu_N}((a-b)^2,a-b)}_{=:A_N}
			   + 2 \underbrace{\Cov_{\mu_N}(ab,a-b)}_{=:B_N}
			   + \underbrace{\Cov_{\mu_N}(b^2,a-b)}_{=:C_N} .
		\end{align*}
		
		From \eqref{sanderson} it is straightforward to show that $A_N = \frac N4$.		
		Note that the symmetry of $\mu_N$ gives
		\begin{equation} \label{covariance zero}
			C_N = \Cov_{\mu_N}(b^2 ,a-b) = 0 ,
		\end{equation}
		which is a straightforward generalization of \cite[Lemma 3.7]{bk10}
		(our \autoref{covariance vanishes by symmetry} below contains a general formulation of this phenomenon).
		Finally, for $B_N$, write
		\[
			B_N
			= \Cov_{\mu_N}(ab,a-b)
			= \E{\mu_N}{ab(a-b)} - \E{\mu_N}{ab} \E{\mu_N}{a-b} .
		\]
		Then
		\begin{align*}
			\E{\mu_N}{ab(a-b)}
				& = \E{\mu_N}{((a-b)+b)b(a-b)} \\
				& = \E{\mu_N}{(a-b)^2b} + \E{\mu_N}{b^2(a-b)} \\
				&\annrel{\eqref{covariance zero}}{=}
				\E{\mu_N}{(a-b)^2b} + \E{\mu_N}{b^2}  \E{\mu_N}{a-b},
		\end{align*}
		and
		\begin{align*}
			\E{\mu_N}{ab} \E{\mu_N}{a-b}
				& = \E{\mu_N}{((a-b)+b)b} \E{\mu_N}{a-b} \\
				& = (\E{\mu_N}{(a-b)b} + \E{\mu_N}{b^2})  \E{\mu_N}{a-b} \\
				&\annrel{\eqref{covariance zero}}{=}
				(\E{\mu_N}{a-b} \cdot \E{\mu_N}{b} + \E{\mu_N}{b^2})  \E{\mu_N}{a-b} \\
				& = \E{\mu_N}{a-b}^2 \cdot \E{\mu_N}{b} + \E{\mu_N}{b^2} \E{\mu_N}{a-b}
		\end{align*}
		yield
		\begin{align*}		
			B_N & = \E{\mu_N}{(a-b)^2 b} - \E{\mu_N}{a-b}^2  \E{\mu_N}{b} \\
				&=
					\underbrace{\Var_{\mu_N}(a-b)}_{\text{$\frac N4$ by \eqref{sanderson}}}
					\underbrace{\E{\mu_N}{b}}     _{\makebox{\makebox[3em][l]{\scriptsize $\frac{(N-1)(N+2)}8$ by \cite{bk10}}}}
					+ \Cov_{\mu_N}(b,(a-b)^2) .
		\end{align*}
		Substituting $A_N$, $B_N$, $C_N$ above yields the proposition.
	\end{proof}

\section{Palindromicity of the string functions}
\label{sec:palindromicity}

	Let us describe a symmetry property of level $1$ Demazure modules, which is crucial for us.
	\begin{lem}
		\label{strings are symmetric}
		Let $\lambda \in \Lambda_0 + \mathbf{Z} \alpha_0 + \mathbb{Z} \alpha_1$.
		Then
		\[
			\dim V_{w_N}(\Lambda_0)_{\lambda}
			= \dim V_{w_N}(\Lambda_0)_{\lambda - S(N,\lambda)\delta} ,
		\]
		where
		\[
			S(N,\lambda) = \begin{cases}
				\frac 14 N^2 + (a-b)^2(\lambda) - 2a(\lambda)
				& \text{if $N$ is even,} \\
				\frac 14 (N^2 -1) + (a-b)^2(\lambda) - (a-b)(\lambda)`
				& \text{if $N$ is odd.}
			\end{cases}
		\]
	\end{lem}

	\begin{proof}
		Recall the definition of the bivariate Macdonald polynomials $P_\nu(z_1, z_2; q, t) \in \mathbf{Q}(q,t)[z_1,z_2]^{S_2}$, where $\nu$ is a partition into at most $2$ parts (see \cite[Chapter VI]{MR1354144}).
		We only use the specializations $P_\nu(z, z^{-1}; q, 0) \in \mathbf{Q}[z^\pm ,q^\pm]$.
		If $\nu$ is the one row partition $\nu = (N, 0)$, then \cite[Theorem 6 and Theorem 7]{MR1771615} gives
		\begin{equation} \label{sanderson macdonald ohne sigma}
			\operatorname{ch}(V_{w_N}(\Lambda_0))
			= e^{\Lambda_0 - \lfloor \frac 14 N^2 \rfloor \delta} \cdot P_\nu(e^{\frac 12\alpha_1}, e^{-\frac 12\alpha_1}; e^\delta, 0) .
		\end{equation}
		For $N \geq k \geq 0$ and $i \in \{ 0, \ldots, k(N-k) \}$ let $s_{N,k,i}$ denote the number of lattice paths in $\mathbf{Z}^2$ from $(0,0)$ to $(k, N-k)$, where each segment of the path is either $(1,0)$ or $(0,1)$, and such that the area under the path is $i$.
		As explained in \cite{MR0269521}, the Gaussian binomial coefficient with parameters $N$, $k$ in a variable $q$ is
		\[
			\qbinom Nk_q
			= \sum_{i=0}^{k(N-k)} s_{N,k,i} q^i .
		\] 
		By \cite[(3.4)]{MR1457389}
		\begin{equation}
			\label{hikami macdonald one}
			P_\nu(z, z^{-1}; q, 0)
			= \sum_{k=0}^N \qbinom Nk_q z^{2k-N} .
		\end{equation}
		Combining \eqref{sanderson macdonald ohne sigma} and \eqref{hikami macdonald one} we obtain
		\begin{align}
			\label{relation q and demazure character}
			\operatorname{ch}(V_{w_N}(\Lambda_0))
				= e^{\Lambda_0 - \lfloor \frac 14 N^2 \rfloor \delta} \cdot
					\sum_{k=0}^N \qbinom Nk_{e^\delta} e^{(2k-N)\frac 12 \alpha_1}
			.
		\end{align}

		By considering the point reflection in $\left( \frac k2, \frac{N-k}2 \right)$, it follows immediately from the above description that the Gaussian binomial coefficients are palindromic, $\qbinom Nk_q = q^{k(N-k)} \qbinom Nk_{q^{-1}}$ and our lemma follows by straightforward computation.
	\end{proof}

\section{The stretching trick}
\label{sec:stretching}
	We want to establish a proper recurrence relation based on the equations described in \autoref{specialized fake recursion}. This section describes how the palindromicity of the string functions allows us to compute the values of $\Cov_{\mu_N}(\nicedot,(a-b)^2)$ at $b$ and $a$, respectively.
	
	The following two propositions follow directly from the definitions.
	
	\begin{prp} \label{measure stretching}
		Let $\mu \in \Meas(\mathbf{R}^2)$.
		Let $X, Y : \mathbf{R}^2 \to \mathbf{R}$ be the projection on the first and second component, respectively.
		Let $q : \mathbf{R}^2 \to \mathbf{R}^2$ be given by $q(x,y) = (x^2, y)$.
		Then
		\[
			\Cov_\mu(X^2,Y) = \Cov_{q_*\mu}(X,Y) .
		\]
	\end{prp}

	By saying that $s : \mathbf{R}^2 \to \mathbf{R}^2$ is a \emph{reflection at $H_1$ along $H_2$} we mean that $H_1$ is the $1$-eigen space and $H_2$ is the $(-1)$-eigen space of $s$.
	A measure $\mu \in \Meas(\mathbf{R}^2)$ is said to be \emph{symmetric} with respect to $s$ if $s_*\mu = \mu$.
	
	\begin{prp} \label{covariance vanishes by symmetry}
		Let $X, Y : \mathbf{R}^2 \to \mathbf{R}$ be the projection on the first and second component, respectively.
		Let $\mu \in \Meas(\mathbf{R}^2)$ be symmetric at $\{Y=0\}$ along $\{X=0\}$.
		Then $\Cov_\mu(X,Y) = 0$.
	\end{prp}

	Now we are ready to exploit the symmetry of the string functions associated with the Demazure module $V_{w_N}(\Lambda_0)$.
	
	\begin{lem} \label{orthogonality wrt stretched measure}
		Let $\mu_N$ be the weight distribution of the Demazure module $V_{w_N}(\Lambda_0)$ supported on the lattice $\Gamma_0 = \Lambda_0 + \mathbf{Z} \alpha_0 + \mathbf{Z} \alpha_1$.
		Define coordinates $X, Y: \Gamma_0 \to \mathbf{R}$ as follows:
		Let
		\[
			X = \begin{cases}
				a-b-\frac 12	 & \text{if $N$ is odd,} \\
				a-b          & \text{if $N$ is even,}
			\end{cases}
			\quad
			Y = \begin{cases}
				b - \frac {N^2-2}8	 & \text{if $N$ is odd,} \\
				a - \frac {N^2}8   & \text{if $N$ is even.}
			\end{cases}
		\]
		Let $q : \Gamma_0 \to \Gamma_0$ such that $X(q(\lambda)) = X(\lambda)^2$ and $Y(q(\lambda)) = Y(\lambda)$ for all $\lambda \in \Gamma_0$.
		Then
		\[
			\Cov_{q_*\mu_N}(X-2Y, X) = 0 .
		\]
	\end{lem}
	
	While reading the proof, see \autoref{string-symmetry-parabola} and \ref{string-symmetry-line} for an illustration.
	
	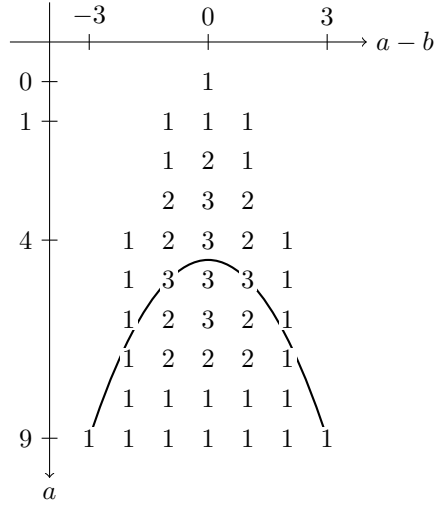
\begin{figure}
		\centering
		\begin{tikzpicture}[x=1.5em,y=-1.5em]
			\draw[->] (-5,-1) -- (4,-1) node[right]{$a-b$};
			\draw[->] (-4,-2) -- (-4,10) node[below]{$a$};
			
			\draw (-3,-.8) -- +(0,-.4) node[above]{$-3$};
			\draw (0,-.8) -- +(0,-.4) node[above]{$0$};
			\draw (3,-.8) -- +(0,-.4) node[above]{$3$};
			
			\draw (-3.8,0) -- +(-.4,0) node[left]{$0$};
			\draw (-3.8,1) -- +(-.4,0) node[left]{$1$};
			\draw (-3.8,4) -- +(-.4,0) node[left]{$4$};
			\draw (-3.8,9) -- +(-.4,0) node[left]{$9$};
			
			\draw[thick] (-3,9) parabola bend (0,4.5) (3,9);
			
			\draw ( 0,0) node[fill=white,inner sep=0.2ex]{$1$};
			
			\draw (-1,1) node[fill=white,inner sep=0.2ex]{$1$};	
			\draw ( 0,1) node[fill=white,inner sep=0.2ex]{$1$};
			\draw ( 1,1) node[fill=white,inner sep=0.2ex]{$1$};
			
			\draw (-1,2) node[fill=white,inner sep=0.2ex]{$1$};
			\draw ( 0,2) node[fill=white,inner sep=0.2ex]{$2$};
			\draw ( 1,2) node[fill=white,inner sep=0.2ex]{$1$};
			
			\draw (-1,3) node[fill=white,inner sep=0.2ex]{$2$};	
			\draw ( 0,3) node[fill=white,inner sep=0.2ex]{$3$};	
			\draw ( 1,3) node[fill=white,inner sep=0.2ex]{$2$};
				
			\draw (-2,4) node[fill=white,inner sep=0.2ex]{$1$};	
			\draw (-1,4) node[fill=white,inner sep=0.2ex]{$2$};	
			\draw ( 0,4) node[fill=white,inner sep=0.2ex]{$3$};	
			\draw ( 1,4) node[fill=white,inner sep=0.2ex]{$2$};	
			\draw ( 2,4) node[fill=white,inner sep=0.2ex]{$1$};
				
			\draw (-2,5) node[fill=white,inner sep=0.2ex]{$1$};	
			\draw (-1,5) node[fill=white,inner sep=0.2ex]{$3$};	
			\draw ( 0,5) node[fill=white,inner sep=0.2ex]{$3$};	
			\draw ( 1,5) node[fill=white,inner sep=0.2ex]{$3$};	
			\draw ( 2,5) node[fill=white,inner sep=0.2ex]{$1$};
				
			\draw (-2,6) node[fill=white,inner sep=0.2ex]{$1$};	
			\draw (-1,6) node[fill=white,inner sep=0.2ex]{$2$};	
			\draw ( 0,6) node[fill=white,inner sep=0.2ex]{$3$};	
			\draw ( 1,6) node[fill=white,inner sep=0.2ex]{$2$};	
			\draw ( 2,6) node[fill=white,inner sep=0.2ex]{$1$};
				
			\draw (-2,7) node[fill=white,inner sep=0.2ex]{$1$};	
			\draw (-1,7) node[fill=white,inner sep=0.2ex]{$2$};	
			\draw ( 0,7) node[fill=white,inner sep=0.2ex]{$2$};	
			\draw ( 1,7) node[fill=white,inner sep=0.2ex]{$2$};	
			\draw ( 2,7) node[fill=white,inner sep=0.2ex]{$1$};
				
			\draw (-2,8) node[fill=white,inner sep=0.2ex]{$1$};	
			\draw (-1,8) node[fill=white,inner sep=0.2ex]{$1$};	
			\draw ( 0,8) node[fill=white,inner sep=0.2ex]{$1$};	
			\draw ( 1,8) node[fill=white,inner sep=0.2ex]{$1$};	
			\draw ( 2,8) node[fill=white,inner sep=0.2ex]{$1$};
				
			\draw (-3,9) node[fill=white,inner sep=0.2ex]{$1$};	
			\draw (-2,9) node[fill=white,inner sep=0.2ex]{$1$};	
			\draw (-1,9) node[fill=white,inner sep=0.2ex]{$1$};	
			\draw ( 0,9) node[fill=white,inner sep=0.2ex]{$1$};	
			\draw ( 1,9) node[fill=white,inner sep=0.2ex]{$1$};	
			\draw ( 2,9) node[fill=white,inner sep=0.2ex]{$1$};	
			\draw ( 3,9) node[fill=white,inner sep=0.2ex]{$1$};	
		\end{tikzpicture}
		\caption{Weight distribution $\mu_6$ of $V_{w_6}(\Lambda_0) = V_{(s_1s_0)^3}(\Lambda_0)$ and the parabola of the string symmetry points.}
		\label{string-symmetry-parabola}
	\end{figure}

	\begin{figure}
		\centering
		\begin{tikzpicture}[x=1.5em,y=-1.5em]
			\draw[->] (-2,-1) -- (10,-1) node[right]{$a-b$};
			\draw[->] (-1,-2) -- (-1,10) node[below]{$a$};
			
			\draw (0,-.8) -- +(0,-.4) node[above]{$0$};
			\draw (1,-.8) -- +(0,-.4) node[above]{$1$};
			\draw (4,-.8) -- +(0,-.4) node[above]{$4$};
			\draw (9,-.8) -- +(0,-.4) node[above]{$9$};
			
			\draw (-0.8,0) -- +(-.4,0) node[left]{$0$};
			\draw (-0.8,1) -- +(-.4,0) node[left]{$1$};
			\draw (-0.8,4) -- +(-.4,0) node[left]{$4$};
			\draw (-0.8,9) -- +(-.4,0) node[left]{$9$};
			
			\draw[thick] (-1,4) -- (10,9.5);
			
			\draw ( 0,0) node[fill=white,inner sep=0.2ex]{$1$};
			
			\draw ( 0,1) node[fill=white,inner sep=0.2ex]{$1$};
			\draw ( 1,1) node[fill=white,inner sep=0.2ex]{$2$};
			
			\draw ( 0,2) node[fill=white,inner sep=0.2ex]{$2$};
			\draw ( 1,2) node[fill=white,inner sep=0.2ex]{$2$};
			
			\draw ( 0,3) node[fill=white,inner sep=0.2ex]{$3$};	
			\draw ( 1,3) node[fill=white,inner sep=0.2ex]{$4$};
				
			\draw ( 0,4) node[fill=white,inner sep=0.2ex]{$3$};	
			\draw ( 1,4) node[fill=white,inner sep=0.2ex]{$4$};	
			\draw ( 4,4) node[fill=white,inner sep=0.2ex]{$2$};
				
			\draw ( 0,5) node[fill=white,inner sep=0.2ex]{$3$};	
			\draw ( 1,5) node[fill=white,inner sep=0.2ex]{$6$};	
			\draw ( 4,5) node[fill=white,inner sep=0.2ex]{$2$};
				
			\draw ( 0,6) node[fill=white,inner sep=0.2ex]{$3$};	
			\draw ( 1,6) node[fill=white,inner sep=0.2ex]{$4$};	
			\draw ( 4,6) node[fill=white,inner sep=0.2ex]{$2$};
				
			\draw ( 0,7) node[fill=white,inner sep=0.2ex]{$2$};	
			\draw ( 1,7) node[fill=white,inner sep=0.2ex]{$4$};	
			\draw ( 4,7) node[fill=white,inner sep=0.2ex]{$2$};
				
			\draw ( 0,8) node[fill=white,inner sep=0.2ex]{$1$};	
			\draw ( 1,8) node[fill=white,inner sep=0.2ex]{$2$};	
			\draw ( 4,8) node[fill=white,inner sep=0.2ex]{$2$};
				
			\draw ( 0,9) node[fill=white,inner sep=0.2ex]{$1$};	
			\draw ( 1,9) node[fill=white,inner sep=0.2ex]{$2$};	
			\draw ( 4,9) node[fill=white,inner sep=0.2ex]{$2$};	
			\draw ( 9,9) node[fill=white,inner sep=0.2ex]{$2$};	
		\end{tikzpicture}
		\caption{Stretched weight distribution $q_* \mu_6$ of $V_{w_6}(\Lambda_0) = V_{(s_1s_0)^3}(\Lambda_0)$ and the line of the string symmetry points.}
		\label{string-symmetry-line}
	\end{figure}
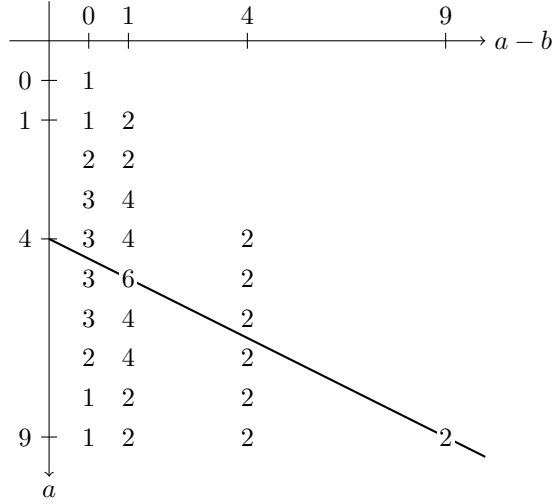

\begin{proof}
		If $N$ is odd, then by \autoref{strings are symmetric} the strings in $\mu_N$ are symmetric around
		\begin{align*}
			b & = \frac 12\left(\frac 14 (N^2-1) + (a-b)^2 - (a-b)\right) \\
				& = \frac 12\left(\frac 14 (N^2-1) + (a-b-\frac 12)^2 - \frac 14 \right) \\
				& = \frac 18 (N^2 - 2) + \frac 12(a-b-\frac 12)^2.
		\end{align*}
		In other words, they are symmetric around $Y = \frac 12 X^2$.
		If $N$ is even, then by \autoref{strings are symmetric} the strings in $\mu_N$ are symmetric around
		\[
			a = \frac 12\left(\frac 14 N^2 + (a-b)^2 \right) = \frac 18 N^2 + \frac 12(a-b)^2.
		\]
		In other words, they are symmetric around $Y = \frac 12 X^2$.
		In both cases, the string midpoints in $\mu_N$ are $(x,\frac 12 x^2)$ in coordinates $X, Y$, so the string midpoints in $q_*\mu_N$ are $(x^2, \frac 12 x^2)$.
		Hence $q_*\mu_N$ is symmetric at $\{X-2Y = 0\}$ along \{X=0\}.
		The lemma follows by \autoref{covariance vanishes by symmetry}.
	\end{proof}
	
	By formulating \autoref{orthogonality wrt stretched measure} in terms of the weight distribution of  $V_{w_N}(\Lambda_0)$ via \autoref{measure stretching} we obtain:

	\begin{cor} \label{intermediate covariance}
	Let $\mu_N$ be the weight distribution of the Demazure module $V_{w_N}(\Lambda_0)$.
	If $N$ is odd, then
		\[
			\Cov_{\mu_N}((a-b)^2 - (a-b)-2b, (a-b)^2 - (a-b)) = 0 .
		\]
	If $N$ is even, then
		\[
			\Cov_{\mu_N}((a-b)^2-2a, (a-b)^2) = 0 .
		\]
		
	\end{cor}
	
	Consequently we can now determine the values of $\Cov_{\mu_N}(\nicedot,(a-b)^2)$ at $b$ and $a$, respectively:
	
	\begin{lem}
	\label{determine linear form}
		Let $\mu_N$ be the weight distribution of the Demazure module $V_{w_N}(\Lambda_0)$. Then,
		\[
			\frac{N(N-1)}{16} = \begin{cases}
				\Cov_{\mu_N}(b,(a-b)^2) & \text{if $N$ is odd,} \\
				\Cov_{\mu_N}(a,(a-b)^2) & \text{if $N$ is even.}
			\end{cases}
		\]
	\end{lem}
	
	\begin{proof}
		By \autoref{intermediate covariance} we obtain for odd $N$
		\begin{align*}
			\Cov_{\mu_N}(b,(a-b)^2) & = \frac 12 \big(\Var_{\mu_N}((a-b)^2) - 2 \Cov_{\mu_N}((a-b)^2,a-b) \\
				& \quad + \Var_{\mu_N}(a-b) + 2\Cov_{\mu_N}(b, a-b) \big),
		\end{align*}
		and for even $N$
		\[
			\Cov_{\mu_N}(a,(a-b)^2) = \frac 12 \Cov_{\mu_N}((a-b)^2, (a-b)^2) = \frac 12 \Var_{\mu_N}((a-b)^2) .
		\]
		We know the values on the right-hand sides of both equations. Let us recollect them in either case. 
		If $N$ is odd, we know by \cite[Lemma 3.7]{bk10} that
		$
			\Cov_{\mu_N}(b, a-b) = 0,
		$
		and from \eqref{sanderson} one can deduce
		\[
			\Cov_{\mu_N}((a-b)^2,a-b) = \frac N4 \quad \text{and} \quad \Var_{\mu_N}(a-b)  = \frac N4 .
		\]
		Hence,
		\[
			\Cov_{\mu_N}(b,(a-b)^2) = \frac 12 \bigg( \frac{N(N+1)}{8} - 2 \frac N4 + \frac N4 - 2\cdot0 \bigg) = \frac{N(N-1)}{16}.
		\]
		If $N$ is even, we again apply \eqref{sanderson} to derive
		\[
			\Var_{\mu_N}((a-b)^2) = \frac{N(N-1)}{8}.
			\qedhere
		\]
	\end{proof}
	
\section{Variance of the degree distribution}
\label{sec:variance}
	
	\autoref{specialized fake recursion} and \autoref{determine linear form} immediately give:
	
	\begin{lem}[Recurrence relations]
		\label{recurrence relation}
			Let $\mu_N$ be the weight distribution of the Demazure module $V_{w_N}(\Lambda_0)$. Then,
		\begin{align*}
			\E{\mu_{N+1}}{a^2} & = \E{\mu_N}{a^2} + \frac{N^2(N+3)}{16} && \text{if $N$ is odd,} \\
			\E{\mu_{N+1}}{b^2} & = \E{\mu_N}{b^2} + \frac{N(N^2+3N-2)}{16} && \text{if $N$ is even.}
		\end{align*}
	\end{lem}
	
	In order to resolve the recurrence relations, we need to switch between the coordinates $a$ and $b$ depending on the parity of $N$.
	Therefore, the following version of \autoref{recurrence relation} is more practical.
	
	\begin{lem}[Modified recurrence relations]
	\label{modified recurrence relation}
		Let $\mu_N$ be the weight distribution of the Demazure module $V_{w_N}(\Lambda_0)$. Then,
		\begin{align*}
			\E{\mu_{N+1}}{a^2}
			& = \E{\mu_N}{b^2} + \frac{N(N+2)(N+3)}{16}
			&& \text{if $N$ is odd,} \\
			\E{\mu_{N+1}}{b^2}
			& = \E{\mu_N}{a^2} + \frac{N(N+1)(N+2)}{16}
			&& \text{if $N$ is even.}
		\end{align*}
	\end{lem}
	
	\begin{proof}
		Write $a^2-b^2=(a-b)(a+b)$ and consider
		\begin{align*}
			\Cov_{\mu_N}(a-b,a+b) & = \E{\mu_N}{a^2-b^2} - \E{\mu_N}{a-b} \E{\mu_N}{a+b}.
		\end{align*}
		For odd $N$ we obtain
		\begin{align*}
			\E{\mu_N}{a^2} & = \E{\mu_N}{b^2} + \Cov_{\mu_N}(a-b,a) + \Cov_{\mu_N}(a-b,b)
				+ \E{\mu_N}{a-b} \E{\mu_N}{a+b} \\
				& = \E{\mu_N}{b^2} + \frac N4 + 0 + \frac 12 \left( \frac 12 + 2 \frac{(N-1)(N+2)}{8} \right) \\
				& = \E{\mu_N}{b^2} + \frac{N(N+3)}{8}
		\end{align*}
		by \cite[Lemma 3.7, Theorem 4.1]{bk10}, and hence
		\[
			\E{\mu_{N+1}}{a^2}
				= \E{\mu_N}{b^2} + \frac{N(N+3)}{8} +  \frac{N^2(N+3)}{16} 
				= \E{\mu_N}{b^2} + \frac{N(N+2)(N+3)}{16}.
		\]
		For even $N$ one similarly derives
		$
			\E{\mu_N}{b^2} = \E{\mu_N}{a^2} + \frac N4,
		$
		and consequently
		\[
			\E{\mu_{N+1}}{b^2}
				= \E{\mu_N}{a^2} + \frac N4 +  \frac{N(N^2+3N-2)}{16} 
				= \E{\mu_N}{a^2} + \frac{N(N+1)(N+2)}{16}.
				\qedhere
		\]
	\end{proof}
	
	The (modified) recurrence relations give: 

\begin{cor}
	\label{degree variance}
		Let $\mu_N$ be the weight distribution of the Demazure module $V_{w_N}(\Lambda_0)$. Then, for $N \geq 1$ we have
		\[
			\frac{N(N-1)(2N+5)}{96} = \begin{cases}
				\Var_{\mu_N}(a) & \text{if $N$ is even,} \\
				\Var_{\mu_N}(b) & \text{if $N$ is odd.}
			\end{cases}
		\]
	\end{cor}
	
	\begin{proof}
		Solving the modified recurrence relations in \autoref{modified recurrence relation} yields
		\begin{align*}
			\E{\mu_N}{a^2} &= \frac{1}{16} \sum\limits_{i=0}^{\frac N2 -1} 2i(2i+1)(2i+2) + \frac{1}{16} \sum\limits_{j=0}^{\frac N2 -1} (2t+1)(2t+3)(2t+4) \\
				& = \frac{N(3N^3+10N^2+9N-10)}{192}
		\end{align*}
		for even $N$, and
		\begin{align*}
			\E{\mu_N}{b^2} &= \frac{1}{16} \sum\limits_{i=0}^{\frac{N-1}2} 2i(2i+1)(2i+2) + \frac{1}{16} \sum\limits_{j=0}^{\frac{N-3}2} (2t+1)(2t+3)(2t+4) \\
				& = \frac{(N-1)(3N^3+13N^2+10N-12)}{192}
		\end{align*}
		for odd $N$.
		Hence, for even $N$ we get
		\begin{align*}
			\Var_{\mu_N}(a)
			&= \E{\mu_N}{a^2} - \E{\mu_N}{a}^2 \\
			&= \frac{N(3N^3+10N^2+9N-10)}{192} - \left( \frac{N(N+1)}8 \right)^2 \\
			&= \frac{N(N-1)(2N+5)}{96} .
		\end{align*}
		Similarly, for odd $N$,
		\begin{align*}
			\Var_{\mu_N}(b)
			&= \E{\mu_N}{b^2} - \E{\mu_N}{b}^2 \\
			&= \frac{(N-1)(3N^3+13N^2+10N-12)}{192} - \left( \frac{(N-1)(N+2)}8 \right)^2 \\
			&= \frac{N(N-1)(2N+5)}{96} . \qedhere
		\end{align*}
	\end{proof}
		
\section{Covariance of the weight distribution}
\label{covariance section}

	For a distribution $\mu \in \Measc(\mathbf{Z}^2)$ and coordinates $X,Y : \mathbf{Z}^2 \to \mathbf{Z}$ we define the \textbf{covariance matrix} of $X$ and $Y$ with respect to $\mu$ to be the $2 \times 2$ matrix
	\[
		\begin{pmatrix}
			\Cov_\mu (X,X) &
			\Cov_\mu (X,Y) \\
			\Cov_\mu (Y,X) &
			\Cov_\mu (Y,Y)
		\end{pmatrix} .
	\]
	
	\begin{thm}[Covariance of the weight distribution]
		\label{covariance of the weight distribution}		
Let $j \in \{0, 1\}$ and $w \in W^\mathrm{aff}$ such that the length $l(ws_j) < l(w) = N$.	
	Then the covariance matrix $\Sigma$ of the degree $\langle -d, \nicedot \rangle$ and the finite weight $\langle \alpha_1^\vee, \nicedot \rangle$ in $V_{w}(\Lambda_j)$ is given by
	\[
		\Sigma = \begin{cases}
			\begin{pmatrix}
				\frac{N(N-1)(2N+5)}{96} & 0 \\
			  0                       & N
			\end{pmatrix}
			& \text{if $N \equiv j \mod (2)$,} \\
			\begin{pmatrix}
				\frac{N(N-1)(2N+5)}{96} + \frac N4 &	 \frac N2\\
				\frac N2                           & N
			\end{pmatrix}
			& \text{if $N \not\equiv j \mod (2)$.}
		\end{cases}			
	\]
	\end{thm}

	\begin{proof}
		We first consider the case that $N \equiv j \mod (2)$.
		The variance of the finite weight follows by \eqref{sanderson}, and the fact that degree and finite weight are uncorrelated follows from the symmetry of the weight distribution.
		The variance of the degree is given by \autoref{degree variance}, where for odd $N$ we use the nontrivial automorphism of the Dynkin diagram so switch from $V_{w_{N,0}}(\Lambda_0)$ to 	$V_{w_{N,1}}(\Lambda_1)$.
			
		As $V_{w_{N,0}}(\Lambda_0)$ corresponds to $V_{w_{N,1}}(\Lambda_1)$ under the nontrivial automorphism of the Dynkin diagram, the result for $N \not\equiv j \mod (2)$ follows by a change of coordinates.
	\end{proof}
	
	For visualization purposes, it is convenient to represent the covariance matrix by the associated \textbf{covariance ellipse}, defined as follows:
	Let $\mu$ be a measure on $\mathbf{R}^2$ with nondegenerate covariance matrix $\Sigma$.
	Then the covariance ellipse of $\Sigma$ is
	\[
		S_\mu = \{ x \in \mathbf{R}^2 : x^t \Sigma^{-1} x = 1 \} .
	\]	
	In \autoref{example figures}, the covariance ellipses have been translated to be centered at the expected weight.
	
\section{Law of large numbers}
\label{sec:wlln}

%
	
	We write $\nu_N \weakto \nu$ if the sequence of measures $\nu_N$ converges weakly to $\nu$ as $N \rightarrow \infty$.
	We will use the following abstract version of the weak law of large numbers, which can be derived from Chebychev's inequality (see e.g.\ \cite[(5.32)]{MR1324786}).
	
	\begin{prp}
		\label{abstract wlln}
		Let $(P_N)$ be a sequence of probability distributions on $\mathbf{R}$ such that $\E{}{P_N} \to c \in \mathbf{R}$ and	$\Var_{}(P_N) \to 0$. Then, $P_N \weakto \delta_c$.
	\end{prp}
	
	Finally,	
	\begin{thm}[Weak law of large numbers]
		\label{wlln}	
		Let $(\Lambda^{(k)})$ be a sequence in $\{\Lambda_0, \Lambda_1\}$, and $(w^{(k)})$ a sequence in $W^\mathrm{aff}$ such that $l(w^{(k)}) \to \infty$.
		Let $\tilde\mu^{(k)} \in \Meas(\mathbf{R}^2)$ be the joint distribution of the degree and the finite weight in $V_{w^{(k)}}(\Lambda^{(k)})$, normalized to a probability distribution and rescaled individually in the two coordinates such that $\supp(\tilde\mu^{(k)})$ just fits into the rectangle $[0,1] \times [-1,1]$.
		Then	, as $k \to \infty$,
		\[
			\tilde \mu^{(k)} \weakto \delta_{(\frac 12, 0)} .
		\]
	\end{thm}
	
	See \autoref{fig wlln} and \ref{degree distributions for Lambda_0} for an illustration.
	
	\begin{figure}
	\label{degree distributions for Lambda_0}
		\includegraphics{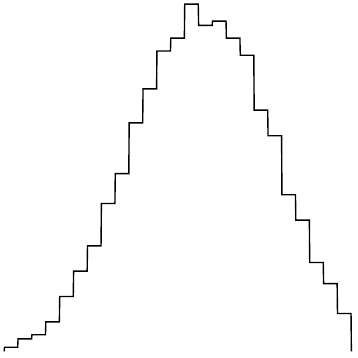}
		\includegraphics{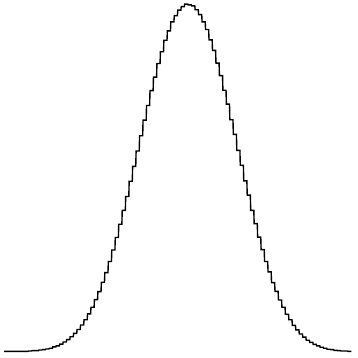} \\
		\includegraphics{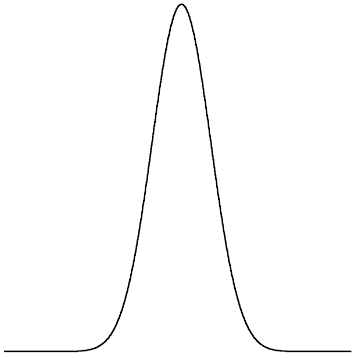}
		\includegraphics{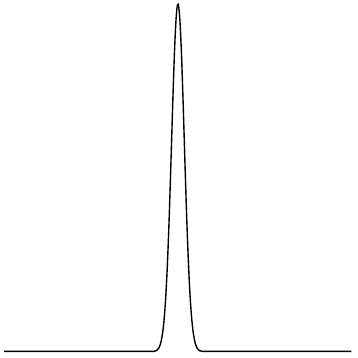}
		\caption{Degree distribution of $V_{(s_1s_0)^k}(\Lambda_0)$ for $k = 5, 10, 25, 500$.}	
	\end{figure}
		
	\begin{proof}
		Let $\mu^{(k)} \in \Meas(\mathfrak{h}^*)$ be the weight distribution of $V_{w^{(k)}}(\Lambda^{(k)})$, and $\hat\mu^{(k)} = (-d, \alpha_0^\vee)_*\mu^{(k)} \in \Meas(\mathbf{N} \times \mathbf{Z})$ the joint distribution of the degree and the finite weight.
		We only consider the sequences $\Lambda^{(k)} = \Lambda_0$ and $w^{(k)} = w_{2k,0} = (s_1s_0)^k$, as the general case follows easily with small modifications.
		Then
		\begin{align*}
			\max\{\langle -d, \lambda \rangle : V_{(s_1s_0)^k}(\Lambda_0)_\lambda \neq 0 \} &= k^2, \\
			\max\{\lvert \langle \alpha_0^\vee, \lambda \rangle \rvert : V_{(s_1s_0)^k}(\Lambda_0)_\lambda \neq 0  \} &= k ,
		\end{align*}
		by \cite[Lemma 4.2]{bk10}. Hence, $\tilde\mu^{(k)}$ is the joint distribution of $-\frac{1}{k^2}d$ and $\frac{1}{k}\alpha_1^\vee$ with respect to $\mu^{(k)}$, normalized to a probability distribution.
			By \cite[Theorem 4.5]{bk10} and \autoref{degree variance} we have
		\begin{align*}
			\E{\mu^{(k)}}{-\frac{1}{k^2}d}
			&= \frac{1}{k^2}\E{\mu^{(k)}}{-d}
			= \frac{1}{k^2} \cdot \frac{k(2k+1)}{4}
			\to \frac 12 , \\
			\Var_{\mu^{(k)}}(-\frac{1}{k^2}d)
			&= \frac{1}{k^4}\Var_{\mu^{(k)}}(-d)
			= \frac{1}{k^4} \cdot \left. \frac{N(N-1)(2N+5)}{96} \right|_{N=2k} \to 0 .
		\end{align*}
		By \autoref{abstract wlln}, $(-\frac{1}{k^2}d)_*\mu^{(k)} \weakto \delta_{\frac 12}$.
		Similarly, $(\frac{1}{k}\alpha_1^\vee)_*\mu^{(k)} \weakto \delta_0$, since
		\begin{align*}
			\E{\mu^{(k)}}{\frac{1}{k}\alpha_1^\vee}
			&= \frac{1}{k} \E{\mu^{(k)}}{\frac{1}{k}\alpha_1^\vee}
			= 0 ,\\
			\Var_{\mu^{(k)}}(\frac{1}{k}\alpha_1^\vee)
			&= \frac{1}{k^2} \Var_{\mu^{(k)}}(\alpha_1^\vee)
			= \frac{1}{k^2} \cdot 2k
			\to 0 .
			\qedhere
		\end{align*}		
	\end{proof}
	
	Let us propose a conjecture about such a concentration in general Demazure modules $V_w (\Lambda)$ with analogously normalized and scaled weight distribution $\tilde\mu_{w,\Lambda}$. By \cite[Corollary 4.3]{bk10} we know that
	\[
		\E{}{\tilde\mu_{w,\Lambda}} \rightarrow \frac{\langle c, \Lambda \rangle +2}{3(\langle c, \Lambda \rangle+1)} \quad \quad \text{as $l(w) \rightarrow \infty$.}
	\]
	
	Consequently,
	
	\begin{cnj}
	\label{cnj-wlln}
Fix a dominant integral weight $\Lambda$ and a sequence $(w^{(k)})$ in $W^{\mathrm{aff}}$ such that $l(w^{(k)}) \to \infty$.
	Let $\mu^{(k)} \in \Meas(\mathbf{N} \times \mathbf{Z})$ be the joint distribution of the degree and the finite weight in $V_{w^{(k)}}(\Lambda)$.
		Let $\tilde\mu^{(k)} \in \Meas(\mathbf{R}^2)$ be the distribution obtained from $\mu^{(k)}$ by normalizing to a probability distribution and rescaling the two coordinates individually so that $\supp(\tilde\mu^{(k)})$ just fits into the rectangle $[0,1] \times [-1,1]$.
		Then, as $k \to \infty$,
		\[
			\tilde\mu^{(k)} \weakto \delta_{\left( \frac{\langle c, \Lambda \rangle +2}{3(\langle c, \Lambda \rangle+1)}, 0 \right)} ,
		\]
		where $c = \alpha_0^\vee + \alpha_1^\vee$ denotes the canonical central element.
	\end{cnj}
	
	This conjecture is further supported by empirical evidence, see \autoref{variance table}. 
	Yet the symmetry property described in \autoref{sec:palindromicity} does not hold for higher level Demazure modules. Therefore, the methods employed in this article, in particular \autoref{sec:stretching}, do not seem to generalize to this case.

	\begin{table}
		\caption{Conjectural variance of the degree $\Var_{\mu_N}(-d)$ in the Demazure module $V_{(s_1 s_0)^k}(m\Lambda_0)$, which is obtained by interpolating values for explicit $N=2k$.
		Comparison with the scaling factor for the degree (see \cite[Lemma 4.2]{bk10}) shows that the law of large numbers holds.}
		\label{variance table}
		\begin{tabular}{c|c|c}
			$m$ & $\Var_{\mu_N}(-d)$                    & $\max \{ \langle -d, \lambda \rangle : V_{(s_1 s_0)^k}(m\Lambda_0)_\lambda \neq 0 \}$ \\ \hline
			& & \\
			$2$ & $\frac{N(N-1)(4N+11)}{81}$    & $\frac 12 N^2$\\
			& & \\
			$3$ & $\frac{N(N-1)(34N+97)}{384}$  & $\frac 34 N^2$ \\			
			& & \\
			$4$ & $\frac{N(N-1)(52N+151)}{375}$ & $N^2$		
		\end{tabular}
	\end{table}

	
	\section{Acknowledgments}
				
		The first author has been supported by the Deutsche Forschungsgemeinschaft, SPP 1388.
		The second author has been supported by the Deutsche Forschungsgemeinschaft, SFB/TR 12.
		
	\bibliographystyle{amsplain}
	\bibliography{degreevariance}

\providecommand{\doi}[1]{\href{http://dx.doi.org/#1}{\nolinkurl{doi:#1}}}
  \providecommand{\arxiv}[1]{\href{http://arxiv.org/abs/#1}{\nolinkurl{arXiv:#%
1}}}
\providecommand{\bysame}{\leavevmode\hbox to3em{\hrulefill}\thinspace}
\providecommand{\MR}{\relax\ifhmode\unskip\space\fi MR }
\providecommand{\MRhref}[2]{%
  \href{http://www.ams.org/mathscinet-getitem?mr=#1}{#2}
}
\providecommand{\href}[2]{#2}
\begin{thebibliography}{10}

\bibitem{MR1324786}
Patrick Billingsley, \emph{Probability and measure}, third ed., Wiley, 1995.

\bibitem{bk10}
Thomas Bliem and Stavros Kousidis, \emph{Expected degree of weights in
  {D}emazure modules of $\widehat{\mathfrak{sl}}_2$}, Transformation Groups
  (2011), \doi{10.1007/s00031-011-9129-6}.

\bibitem{MR1457389}
Kazuhiro Hikami, \emph{Representation of the {Y}angian invariant motif and the
  {M}acdonald polynomial}, J. Phys. A \textbf{30} (1997), 2447--2456,
  \doi{10.1088/0305-4470/30/7/023}.

\bibitem{MR1104219}
Victor Kac, \emph{Infinite-dimensional {L}ie algebras}, third ed., Cambridge
  University Press, 1990.

\bibitem{MR894387}
Shrawan Kumar, \emph{Demazure character formula in arbitrary {K}ac--{M}oody
  setting}, Invent. Math. \textbf{89} (1987), 395--423,
  \doi{10.1007/BF01389086}.

\bibitem{MR1354144}
Ian Macdonald, \emph{Symmetric functions and {H}all polynomials}, second ed.,
  Clarendon Press, 1995.

\bibitem{MR862200}
Olivier Mathieu, \emph{Formules de {D}emazure--{W}eyl, et g\'en\'eralisation du
  th\'eor\`eme de {B}orel--{W}eil--{B}ott}, C. R. Acad. Sci., Paris, S\'er. I
  \textbf{303} (1986), 391--394.

\bibitem{MR980506}
\bysame, \emph{Formules de caract\`eres pour les alg\`ebres de {K}ac--{M}oody
  g\'en\'erales}, Ast\'erisque, no. 159--160, Soci{\'e}t{\'e} Math{\'e}matique
  de France, 1988.

\bibitem{MR0269521}
George P{\'o}lya, \emph{Gaussian binomial coefficients and the enumeration of
  inversions}, Proc. {S}econd {C}hapel {H}ill {C}onf. on {C}ombinatorial
  {M}athematics and its {A}pplications, Univ. North Carolina, Chapel Hill,
  N.C., 1970, pp.~381--384.

\bibitem{MR1407880}
Yasmine Sanderson, \emph{Real characters for {D}emazure modules of rank two
  affine {L}ie algebras}, J. Algebra \textbf{184} (1996), 985--1000,
  \doi{10.1006/jabr.1996.0294}.

\bibitem{MR1771615}
\bysame, \emph{On the connection between {M}acdonald polynomials and {D}emazure
  characters}, J. Algebraic Combin. \textbf{11} (2000), 269--275,
  \doi{10.1023/A:1008786420650}.

\end{thebibliography}

\end{document}